\documentclass[11pt]{article}

\usepackage{amsrefs}
\usepackage{amsmath,amssymb,amsthm,enumerate}
\usepackage[all,cmtip]{xy}
\usepackage[applemac]{inputenc}

\def \1{\mathds{1}}
\def \a{{\mathfrak a}}
\def \al{\alpha}
\def \Asc{\operatorname{Asc}}
\def \be{\beta}

\def \CB{{\cal B}}
\def \CC{{\cal C}}

\def \CF{{\cal F}}
\def \CG{{\cal G}}
\def \CI{\mathcal I}

\def \CO{{\cal O}}

\def \CR{{\cal R}}

\def \coim{\operatorname{coim}}
\def \coker{\operatorname{coker}}

\def \df{\ \begin{array}{c} _{\rm def}\\ ^{\displaystyle =}\end{array}\ }

\def \e{\emph}

\def \F{{\mathbb F}}
\def \full{\mathrm{full}}
\def \Ga{\Gamma}

\def \im{\operatorname{im}}

\def \Mod{\mathrm{Mod}}
\def \N{{\mathbb N}}
\def \ol{\overline}

\def \ph{\varphi}
\def \Q{{\mathbb Q}}
\def \R{{\mathbb R}}

\def \sm{\smallsetminus}
\def \spec{\operatorname{spec}}

\def \Z{{\mathbb Z}}
\def \({\left(}
\def \){\right)}
\def \={\ =\ }

\newcommand{\tto}
[1]{\stackrel{#1}{\longrightarrow}}

\newtheorem{lemma}{Lemma}[section]
\newtheorem{proposition}[lemma]{Proposition}

\newtheorem{exmples}[lemma]{Examples}
\newenvironment{examples}{\begin{exmples}\nopagebreak\begin{itemize}\nopagebreak\rm}{\end{itemize}\end{exmples}}
\newtheorem{exmple}[lemma]{Example}
\newenvironment{example}[0]{\begin{exmple}\rm}
{\end{exmple}}
\newtheorem{defi}[lemma]{Definition}
\newenvironment{definition}[0]{\begin{defi}\rm}
{\end{defi}}

\begin{document}

\pagestyle{myheadings} \markright{COHOMOLOGY OF CONGRUENCE SCHEMES}

\title{Cohomology of congruence schemes}
\author{Anton Deitmar}
\date{}
\maketitle

{\bf Abstract:}
Modules for sesquiads and congruence schemes are introduced.
It is shown that the corresponding categories are belian and that base change functors establish an ascent datum which allows for a cohomology theory  to be established.

$$ $$

\tableofcontents

\newpage
\section*{Introduction}

In recent years, the undertaking of forming a new 
general paradigm underlying algebra, geometry and number theory in the form of $\F_1$-geometry has attracted many mathematicians, who have shown determination in the consensus of establishing as many different theories as possible
\cites{KOW,Soule,Haran,Toen,CC1,Haran2,Lor1,CC2,CC3,CC4,CC5,Lor2,Lor3,Lor4,Lor5,Lor6,Chu}.
A common core to all the different $\F_1$-theories is formed by monoid schemes \cites{Kato,F1,F1-2,Toen,Vezzani}.
In \cite{belian}, the author has established a general framework for cohomology theories in $\F_1$-geometry and has applied it to sheaf cohomology over monoid schemes.
In \cite{congruence} the author has extended the theory of monoid schemes to a flexible and general setting, called congruence schemes.
In the current paper it is verified that the module sheaves over congruence schemes satisfy the conditions of \cite{belian} so that a meaningful cohomology theory is established.
The main technical point is that in the definition of module homomorphism the extra condition of 'fullness' is required to make arguments of homological algebra work.

In the first section of this paper the notion of modules over a sesquiad is given.
A sesquiad is a multiplicative monoid with a partially defined addition.
In order to avoid pathologies, the addition is required to be tame enough as to allow an embedding into a ring.
So a natural definition for a module over a sesquiad is that of a monoid module with a partial addition that is tame enough to allow for an embedding into a ring module.
A more convenient way to formulate this is to make the ring a part of the structure. This is the philosophy followed in the text.
After verifying first properties and defining tensor products, the crucial condition of fullness is introduced.
For a submodule, fullness means that no additive structure is lost. A homomorphism is called full, if it image is full.
It then is shown that the category of modules and full morphisms is belian in the sense of \cite{belian}.
In the second section, there appear strong morphisms, i.e., morphisms for which the canonical map from the coimage to the image is an isomorphism, as is the case in abelian categories.
Strong morphisms have been central in \cite{belian} and there is an interesting interplay between fullness and strongness.
The third section deals with the concepts of flat and etale maps. These will, among other things, be essential for establishing cohomology theories on the corresponding sites.
In the fourth section finally, it is shown that  the module category of a congruence scheme is belian and that the base change functors establish an ascent datum which allows for a cohomology theory to be established.
First properties of the latter are given.

\section{Modules}
The simplest way to define a sesquiad is to say that a sesquiad is a pair $(A,R)$ consisting of a ring $R$ and a multiplicative submonoid $A\subset R$ containing the zero element and generating the ring $R$.

A homomorphism $\ph:(A,R)\to (A',R')$ of sesquiads is a ring homomorphism $\ph:R\to R'$ such that $\ph(A)\subset A'$.

The ambient ring structure then defines a partially defined addition on $A$. For each $n\in\N$ and each $k\in\Z^n$ we define $D_k$ to be the set of all $a\in A^n$ such that $\sum_{j=1}^nk_ja_j\in A$.
Each $k\in\Z^n$ defines the addition $\Sigma_k:D_k\to A$ given by $\Sigma_k(a)=\sum_{j=1}^nk_ja_j$.
The alternative way of viewing a sesquiad is to consider it as a monoid $A$ together with partially defined additions that come from a ring embedding.
Among all rings defining the addition in this way, there is a universal one, written as $R_A$. The map $A\mapsto (A,R_A)$ establishes the equivalence of these two definitions. 

There is the special case of \e{trivial addition}, when $R_A=\Z[A]/\Z 0_A$, where $0_A$ is the zero element of $A$, i.e., the unique element $0_A$ with $a0_A=0_A$ for every $a\in A$, whose existence is insisted upon.

The category RINGS of rings is a subcategory of the category SES of sesquiads via $R\mapsto (R,R)$.

The category MON of commutative monoids with zero is a subcategory of SES via $A\mapsto (A,\Z[A]/\Z 0_A)$.

\begin{definition}
Given a sesquiad $A$, a \e{module} over $A$ is a pair $(S,M_S)$ of an $R_A$-module $M_S$ and a subset $S\subset M_S$ generating the module and being stable under $A$, i.e., one has $AS\subset S$.
So, in particular, $S$ contains the zero element.

A \e{module homomorphism} $(S,M_S)\to (T,M_T)$ is an $A$-module map $f:S\to T$ which extends to an $R_A$-module homomorphism $M_f:M_S\to M_T$.
As $S$ generates $M_S$, the extension $M_f$ is uniquely determined by $f$.
\end{definition}

\begin{example}
If $\ph:A\to B$ is a sesquiad homomorphism, then $B$ is an $A$-module.
\end{example}

\begin{proposition}\label{lem1.3}
Let $f:S\to T$ be a module homomorphism.
\begin{enumerate}[\rm (a)]
\item $f$ is an isomorphism if and only if $f$ is surjective and $M_f$ is injective.
\item $f$ is a monomorphism iff and only if $M_f$ is injective.
\item $f$ is an epimorphism if and only if it $M_f$ surjective.
\end{enumerate}
A module homomorphism can be mono and epi without being an isomorphism.
\end{proposition}

\begin{proof}
(a) Suppose $f$ is an isomorphism.
Then $M_f:M_S\to M_T$,  is an isomorphism of $R_A$-modules, hence bijective.

For the converse direction assume $f$ to be surjective on $S\to T$ and  $M_f$ injective.
As $T$ generates $M_T$, the map $M_f$ is surjective as well, so there is an $R_A$-module homomorphism $M_g:M_T\to M_S$ inverting $M_f$. As $M_f$ maps $S$ onto $T$, the inverse map $M_g$ maps $T$ to $S$ and therefore is an inverse of $A$-modules.

(b) Assume $f$ is mono and let $K=\ker(f)$.
We want to show that $K=0$.
Let $k\in K$ and define a module homomorphism $g:(A,R_A)\to (S,M_S)$ given by $g(1)=k$.
Then $fg$ is the zero morphism. As $f$ is mono, it follows $g=0$, hence $k=0$.
The converse direction is trivial.

Part (c) is easy.
Finally, we give an example of a morphism $f$ which is epi and mono without being iso.
Let $S$ be any $A$-module such that $S\ne M_S$ and let $T=M_S$ and $f:S\to T$ the inclusion. Then $f$ is an isomorphism of $R_A$-modules, but not of $A$-modules.
\end{proof} 

\begin{definition}
We define the \e{tensor product} $S\odot T$ of two $A$-modules $S$ and $T$ to be the set of simple tensors $s\otimes t$ in $M_S\otimes_{R_A}M_T$, where $s\in S$ and $t\in T$.
The corresponding $R_A$-module is
$$
M_{S\odot T}=M_S\otimes M_T,
$$
where the tensor product on the right hand side is over the ring $R_A$.
\end{definition}

Note that if $A$ is a ring, it has generally more modules in the sesquiad sense than in the ring sense.
Also,  if $S$ and $T$ are modules in the ring sense,
the sesquiadic tensor product $S\odot T$ will in general differ from the ring theoretic tensor product $S\otimes T$.

\begin{definition}
Let $A$ be a sesquiad and $S,T,U$ be sesquiad modules.
A map $b:S\times T\to U$ is called \e{bilinear}, if for all $s\in S$, $t\in T$ the maps $b(s,.)$ and $b(.,t)$ are module homomorphisms.
\end{definition}

\begin{proposition}
The natural map $b_0:S\times T\to S\odot T$ mapping $(s,t)$ to the simple tensor $s\odot t=s\otimes t$ is bilinear.
Every bilinear map from $S\times T$ factors uniquely over $b_0$.
\end{proposition}

\begin{proof}
The first assertion is clear.
For the second, let $b:S\times T\to U$ be bilinear.
Define $\al:S\odot T\to U$ by $\al(s\odot t)=b(s,t)$.
We have to show that $\al$ is well-defined.
For this note that $\al$ is the restriction of the  map $M_S\otimes M_T\to M_U$ which is induced by the bilinear extension $M_S\times M_T\to M_U$ of $b$.
So $\al$ is well-defined.
We get $b=\al b_0$ and $\al$ is unique with that property.
\end{proof}

\begin{definition}
Let $T$ be an $A$-module.
For an arbitrary submodule $U$ we define the \e{full closure} to be
$$
\ol U^\full=M_U\cap T.
$$
A submodule $U\subset T$ is called \e{full}, if
$
U=\ol U^\full.
$

Let $\al:S\to T$ be an $A$-module homomorphism.
Then $\al$ is called a \e{full morphism}, if the image $\al(S)$ is a full submodule of $T$, i.e., if
$
\al(S)=\al(M_S)\cap T.
$
\end{definition}

Kernels and cokernels of full morphisms are full.

\begin{proposition}\label{lem1.6}
A module homomorphism $f$ is an isomorphism if and only if $f$ is full and $M_f$ is an isomorphism.
\end{proposition}

\begin{proof}
For the non-trivial direction let $f$ be full and $M_f$ be an isomorphism. Then $f$ is surjective and and $M_f^{-1}|_{T}$ is an inverse to $f$.
\end{proof}

\begin{examples}
\item Let $S$ be an $A$-module such that $S\ne M_S$.
Set $T=M_S$, then $S$ is a submodule of $T$, which is not full.
\item If $S$ and $T$ have trivial addition, then every module homomorphism $S\to T$ is full.
\item Let $T$ have non-trivial addition and assume that the $A$-set $T$ equipped with the trivial addition is also an $A$-module.
Then the identity map $T_{triv}\to T$ is a bijection, but not an isomorphism.
\end{examples}

Recall from \cite{belian}, the a \e{belian category} is a balanced pointed category which contains kernels, cokernels and finite products and has the property that every morphism with zero cokernel is epic.
Let $A$ be a sesquiad and let $\Mod(A)$ be the category of $A$-modules and $A$-homomorphisms.
Further let $\Mod_\full(A)$ be the subcategory of full morphisms.

\begin{proposition}\label{prop1.8}
$\Mod_\full(A)$ is a belian category with enough injectives and projectives.
\end{proposition}

\begin{proof}
To see that $\Mod_\full(A)$ is balanced, let $f:S\to T$ be a full morphism which is monic and epic.
By Lemma \ref{lem1.3} the morphism of $R_A$ modules $M_f$ is monic and epic, hence an isomorphism.
By Lemma \ref{lem1.6}, $f$ is an isomorphism.
The category $\Mod_\full(A)$ is pointed by the zero module. For any morphism $f:S\to T$ the module
$(f^{-1}(0),M_{f^{-1}(0)})$ is a kernel.
Further let $T/f(S)$ be the image of $T$ in $M_T/M_{f(S)}$. Then $(T,M_T)\to(T/f(S),M_T/M_{f(S)})$ is a cokernel for $f$.
A categorial product of two modules $S$ and $T$ is given by their cartesian product.
Let finally $f:S\to T$ be a morphism with zero cokernel.
Then $M_T=M_{f(S)}$ and therefore by Lemma \ref{lem1.3}, the morphism $f$ is an epimorphism.

Every module admits a surjection from a free module, free modules are projective, and surjections are full. Therefore there are enough projectives in $\Mod_\full(A)$.
The existence of enough injectives is a consequence of the following Lemma together with the fact that the category of $R_A$-modules has enough injectives.

\begin{lemma}\label{lem1.9}
A module $(I,M_I)$ is injective if and only if $I=M_I$ and $M_I$ is injective as  $R_A$-module.
\end{lemma}

{\it Proof.} Let $(I,M_I)$ be injective. We consider the injection $(I,M_I)\hookrightarrow (M_I,M_I)$.
By injectivity, the identity $(I,M_I)\tto = (I,M_I)$ extends to a map $(M_I,M_I)\to (I,M_I)$, which implies $I=M_I$, as $M_I$ is generated by $I$.
To see that $M_I$ is injective, note that an injection of $R_A$-modules $M\hookrightarrow N$ induces an injection of $(A,R_A)$ modules $(M,M)\hookrightarrow (N,N)$.
The converse direction is trivial.
\end{proof}

\section{Strong morphisms}
Let $A$ be a sesquiad and let $S$ be an $A$-module.
For a submodule $U\subset S$ we define the quotient to be
$$
S/U=\text{ image of } S \text{ in } M_S/M_U.
$$
So $a,b\in S$ are equal in $S/U$ if and only if $a-b\in M_U$.
The universal module for $S/U$ is $M_S/M_U$.

\begin{lemma}\label{alem2.1}
\begin{enumerate}[\rm (a)]
\item For a submodule $U$ of $S$ we have $\ker(S\to S/U)=\ol U^\full$.
\item A submodule $U\subset M$ is full if and only if it is the kernel of some morphism $S\to T$.
\end{enumerate}
\end{lemma}

\begin{proof}
(a) An element $s$ of $S$ lies in the kernel of $S\to S/U$ if and only if $s\in M_U$.

(b) If $U$ is full, it is a kernel by (a).
For the converse direction, let $U$ be the kernel of $f:S\to T$.
Let $s\in \ol U^\full=M_U\cap S$.
Then $s=\sum_{j=1}^nk_js_j$ for some $k_j\in\Z$ and $s_j\in \ker(f)$.
This means that the sum is defined in $S$ and therefore $f(s)=f\(\sum_{j=1}^nk_js_j\)=\sum_{j=1}^nk_jf(s_j)=0$, so $s\in U$ and $U$ is full.
\end{proof}

\begin{definition}
The category $\Mod(A)$ of all $A$-modules has a zero object, so it makes sense to speak of kernels and cokernels.
It is easy to see that every morphism in $\Mod(A)$ possesses both.
In particular, let $f:S\to T$; then $\coker(f)=T/f(S)$.
We define the \e{image} and \e{coimage} of a morphism $f$ as
\begin{itemize}
\item $\im(f)\df\ker(\coker(f))$,
\item $\coim(f)\df\coker(\ker(f))$.
\end{itemize}
A morphism $f$ is called \e{strong}, if the natural map from $\coim(f)$ to $\im(f)$ is an isomorphism.
\end{definition}

\begin{proposition}\label{lem2.1}
\begin{enumerate}[\rm (a)]
\item A morphism $f:S\to T$ is strong if and only if it is full and 
$$
M_{\ker f}\to M_S\to M_T
$$ 
is exact.
\item For every submodule $U$ of $S$ the map $S\to S/U$ is strong.
\item Kernels and cokernels are strong.
\end{enumerate}
\end{proposition}

\begin{proof}
For (a) let $f:S\to T$ be strong, so $S/\ker(f)=\coker(\ker(f))\to\ker(\coker(f))$ is an isomorphism.
By Lemma \ref{alem2.1} we have
$$
\ker(\coker(f))=\ker(T\to T/f(S))=\ol{f(S)}^\full,
$$
so $f$ is full. Further, by Lemma \ref{lem1.3} the map 
$M_S/M_{\ker(f)}\cong M_{S/\ker(f)}\to M_T$ is injective, which is equivalent to the claimed exactness.
The converse direction uses the same arguments.

(b) Let $f:S\to S/U$ the canonical map.
Then $f$ is surjective, hence full.
Further we have $\ker(f)=\ol{U}^\full=M_U\cap S$ by Lemma \ref{alem2.1}.
Hence $M_{\ker(f)}=M_{\ol U^\full}=M_U$.
Since on te other hand, $M_{S/U}=M_S/M_U$, the sequence $M_{\ker(f)}\to M_S\to M_{S/U}$ equals $M_U\to M_S\to M_S/M_U$, so it is exact.

(c) A cokernel ist strong by (b).
Let $F:S\to T$ and let $\al:R\to s$ be its kernel.
We claim that $\al$ is strong.
Firstly, it is full by Lemma \ref{alem2.1} and secondly, the sequence $M_{\ker(\al)}\to M_R\to M_S$ is $0\to M_{\ker(f)}\to M_S$ which is exact as $\ker(f)$ is a submodule of $S$.
\end{proof}

\begin{lemma}\label{lem2.2}
\begin{enumerate}[\rm (a)]
\item Let $f:S\to T$ and $g: F\to H$ be full morphisms, then $f\odot g:S\odot F\to T\odot H$ is a full morphism.
\item If $f:S\to T$ is a strong morphism, then $1\odot f:F\odot S\to F\odot T$ is a strong morphism.
\end{enumerate}
\end{lemma}

\begin{proof}
(a) In the notation of the lemma we have
\begin{align*}
f\odot g(S\odot F)&=F(S)\odot g(F)\\
&= \(M_{f(S)}\cap T\)\odot\(M_{g(F)}\cap H\)\\
&= \(M_{f(S)}\otimes M_{g(F)}\)\cap\(T\odot H\)\\
&= M_{f(S)\odot g(F)}\cap \( T\odot H\)\\
&= M_{f\odot g(S\odot F)}\cap \(T\odot H\).
\end{align*}

(b) The sequence 
$$
0\to M_{\ker(f)}\to M_S\to M_T
$$
is exact.
Therefore the sequence 
$$
M_F\otimes M_{\ker(f)}\to M_F\otimes M_S\to M_F\otimes M_T
$$ 
is exact. But this is the sequence
$$
M_{F\odot\ker(f)}\to M_{F\odot S}\to M_{F\odot T}.
$$
By (b) we know that $1\odot f$ is full, so by Proposition \ref{lem2.1} the claim follows.\end{proof}

\section{Flat an etale maps}

\begin{definition}
A sequence of morphisms $S\tto\al T\tto\be U$ is called \e{exact at $T$}, if $\be\al=0$ and the ensuing morphism $\im(\al)\to\ker(\be)$ is an isomorphism.
A sequence $\dots\to S_{i-1}\to S_i\to S_{i+1}\to\dots$ is called \e{exact}, if it is exact at every $S_j$.
A sequence is called \e{strong}, if all morphisms in it are strong and \e{strong exact} if it is strong and exact.
\end{definition}

\begin{lemma}\label{lem3.1}
Let $(*)\equiv\ S\tto fT\tto gU$ be a sequence. 
\begin{enumerate}[\rm (a)]
\item If $f$ is strong and $(*)$ is exact, then the ensuing sequence
\begin{align*}
M_S\tto{M_f}M_T\tto{M_g}M_U\tag*{$M(*)$}
\end{align*}
of $R_A$-modules is exact.
\item If $g$ is strong and the sequence $M(*)$ is exact, then $(*)$ is exact.
\end{enumerate}\end{lemma}

\begin{proof}
(a) Assume the sequence $S\tto fT\tto gU$ is exact.
Since $g$ is strong, by Proposition \ref{lem2.1} it induces an isomorphism $M_U\cong M_T/M_{\ker(g)}=M_T/M_{f(S)}=M_T/M_f(M_S)$, which means that the sequence of $R_A$-modules is exact.

(b) Assume $g$ is strong and $M(*)$ exact.
Let $t\in\ker(g)$. Since $M(*)$ is exact, we have $t\in M_f(M_S)=M_{f(S)}$, hence $t\in M_{f(S)}\cap T=f(S)$ since $f$ is strong.
\end{proof}

\begin{lemma}
Let $f:S\to T$ be strong and suppose that $\ker(f)=0$.
Then the sequence
$$
0\to S\tto fT\to T/f(S)\to 0
$$
is strong exact.
\end{lemma}

\begin{proof}
The map $g:T\to T/f(S)$ is strong by Proposition \ref{lem2.1}.
The sequence is clearly exact at $S$ and at $T/f(S)$.
At $T$ we have $\im(f)\cong S$ and $f$ is an isomorphism $S\to f(S)$ as $f$ is strong and has zero kernel.\end{proof}

\begin{lemma}
 Let 
$$
0\to S\tto\al T\tto\beta U\to0
$$ 
be strong exact and $F$ arbitrary.
Then the sequence
$$
F\odot S\to F\odot T\to F\odot U\to 0
$$
is strong exact.
\end{lemma}

\begin{proof} All morphisms are strong.
Therefore $1\odot\beta$ is surjective.
It remains to show that $1\odot\al(F\odot S)=\ker(1\odot\beta)$.
We have $\ker(1\odot\beta)=F\odot\ker(\beta)=F\odot\al(S)=1\odot\al(F\odot S)$.
\end{proof}

\begin{definition}
A module $F$ is called \e{flat}, if for every strong exact sequence
$$
0\to S\to T\to U\to 0
$$
the sequence
$$
0\to F\odot S\to F\odot T\to F\odot U\to 0
$$
is exact.
By Lemma \ref{lem2.2} the latter sequence  is strong.
\end{definition}

\begin{proposition}
\begin{enumerate}[\rm (a)]
\item $F$ is flat if and only if $M_F$ is flat as $R_A$-module.
\item $F$ is flat if and only if for every finitely generated ideal $\a$ of $A$ the map $\a\odot F\to F$ is injective.
\end{enumerate}
\end{proposition}

\begin{proof}
(a) Let  $0\to S\to T\to U\to 0$ be strong exact. Then the sequence $0\to F\odot S\to F\odot T\to F\odot U\to 0$ is strong and by Lemma \ref{lem3.1} it is exact if and only if the sequence
$$
0\to M_F\otimes M_S\to M_F\otimes M_T\to M_F\otimes M_U\to 0
$$
is exact. The latter is the case if $M_F$ is flat.

For the converse direction let $0\to X\tto fY$ be an exact sequence of $R_A$-modules and let $F$ be a flat $A$-module.
Consider $X$ and $Y$ as $A$-modules by leaving the full addition, which means that $M_X=X$ and $M_Y=Y$.
The sequence $0\to F\odot X\to F\odot Y$ is strong exact and by Lemma \ref{lem3.1} the sequence $0\to M_F\otimes X\to M_F \otimes Y$ is exact which means that $M_F$ is flat.
The proof of (b) is similar to the ring case as in Chapter 1 of \cite{Bourb}.
\end{proof}

\begin{definition}
A congruence $C$ on a sesquiad $A$ is called \e{finitely generated} if there are $a_1,\dots,a_n, b_1,\dots b_n\in A$ such that
$$
C=\bigcap_{a_j\sim_E b_j}E,
$$
where the intersection is extended over all congruences $E$ such that $a_1\sim_E b_1,\dots, a_n\sim_E b_n$ holds.
\end{definition}

\begin{lemma}
A congruence $C$ is finitely generated if and only if the corresponding ideal $I(C)\subset R_A$ is finitely generated.
\end{lemma}

\begin{proof}
Let $C$ be finitely generated, say by $a_1\sim b_1,\dots,a_n\sim b_n$.
Then $I(C)$ is generated by the elements $a_1-b_1,\dots,a_n-b_n$.
Conversely, assume that $I(C)$ is finitely generated.
Then any set of generators contains a finite set of generators.
As $I(C)$ is generated by $\{ a-b:a\sim_Cb\}$, there are $a_1,\dots,a_n, b_1,\dots b_n\in A$ such that $I(C)$ is generated by $a_1-b_1,\dots,a_n-b_n$.
Then $C$ is generated by $a_1\sim b_1,\dots,a_n\sim b_n$.
\end{proof}

\begin{definition}
\begin{enumerate}[\rm (a)]
\item A morphism $\ph:A\to B$ of sesquiads is called \e{finite}, resp. \e{of finite type} if $R_B$ is finite resp. of finite type over $R_A$. 
\item A morphism of sesquiads $f:A\to B$ is called \e{finitely presented}, if the ring homomorphism $R_A\to R_B$ is finitely presented.
\end{enumerate}
\end{definition}

\begin{definition}
A sesquiad $A$ is called \e{simple} if every sesquiad morphism $A\to B$ with $B\ne 0$ is injective.
\end{definition}

This is equivalent to 
$$
\spec_c(A)=\{\Delta\},
$$
or to
$A$ being integral and $\spec_cA$ consisting of a single point.

\begin{examples}
\item A ring is a simple sesquiad if and only if it is a field.
\item $\F_1=\{ 0,1\}$ with the trivial addition is simple.
\item $A=\{0,1,-1\}\subset \F_p$ for a prime $p\ge 3$ is simple.
\item Any sesquiad $A$ such that $R_A$ is a field, is simple.
\end{examples}

\begin{definition}
A congruence $C$ is called a \e{maximal congruence}, if it is maximal in the set of all congruences with $0\nsim 1$.
Every maximal congruence is prime. A congruence $C$ is maximal if and only if the quotient $A/C$ is simple.

Let $E$ be a prime congruence of the sesquiad $A$.
The  sesquiad
$$
\kappa(E)= A_E/E_E
$$
is called the \e{residue sesquiad} of $E$.
\end{definition}

\begin{lemma}\label{lem4.5}
Let $A$ be a sesquiad and $C$ a congruence on $A$.
Then $C$ is maximal if and only if $A/C$ is simple.
\end{lemma}

\begin{proof}
Any given congruence $E$ of $A/C$ induces a congruence $D\supset C$ by $D=\ker_c(A\to (A/C)/E)$.
\end{proof}

\begin{example}
Let $A=\{0,1,\tau,\tau^2,\dots\}$ be the free monoid generated by one element.
Then $R_A$ is the polynomial ring $\Z[X]$ and $A_\Delta$ is the submonoid of the quotient field $\Q(X)$ generated by $X,X^{-1}$ and $\pm(1-X^k)^{-1}$, where $k\in\N$.
So each non-zero element of $A_\Delta$ is of the form $X^k$ for some $k\in\Z$ or $\pm\frac{X^k}{(1-X^{l_1})\cdots (1-X^{l_s})}$ with $s,l_1,\dots,l_s\in\N$.
\end{example}

\begin{definition}
A morphism $f:X\to Y$ of congruence schemes is said to be \e{of finite presentation at $x\in X$}, if there exists an open affine neighborhood $U\subset X$ of $x$ and an open affine neighborhood $V\subset Y$ of $f(U)$ such that the ensuing morphism $\CO_Y(V)\to \CO_X(U)$ is of finite presentation.
The Morphism $f$ is \e{locally of finite presentation}, if it is finitely presented at all points of $X$.
\end{definition}

\begin{definition}
A \e{subsesquiad} $A\subset B$ is a subset which is a sesquiad with the structures of $B$.
In particular, the universal ring $R_A$ then is a subring of $R_B$.
We also say that $B$ is an \e{extension} of $A$ and write this as $B/A$.
We say the extension is \e{finite}, \e{of finite type} or \e{of finite presentation}, if the inclusion morphism $A\hookrightarrow B$ has the corresponding property.

Note that for an injective  sesquiad morphism $\ph:A\hookrightarrow B$, the image is a susbsesquiad, which is not necessarily isomorphic to $A$.
This is not even so for simple sesquiads as the example $A=\{ 0,1\}$ with trivial addition and $B=\F_2$ shows.

A subsesquiad $A\subset B$ of a sesquiad $B$ is called a \e{full subsesquiad}, if 
$$
A=R_A\cap B.
$$
\end{definition}

\begin{example}
Let $A$ be a sesquiad with $A\ne R_A$ and let $B=R_A$.
Then $A$ is a subsesquiad of $B$ which is not full.
\end{example}

\begin{proposition}
Let $A$ be a simple sesquiad. Then $A^\times \subset A\sm\{0\}\subset R_A^\times$ and both in inclusions are strict in general.
\end{proposition}

\begin{proof}
The only non-trivial part concerns the second inclusion.
So let $A$ be a sesquiad and  let $a\in A\cap\(R_A\sm R_A^\times\)$. Then the map $A\to R_A/aR_A$ is not injective, therefore $A$ is not simple. This proves the second inclusion.
A counterexample to ``$=$" is given by the sesquiad $A$ of all rational numbers $\frac pq$ with $p=0$ or $p$ odd.
Then $R_A=\Q$ is a field, therefore $A$ is simple.
\end{proof}

\begin{definition}
Let $A$ be a sesquiad. A \e{polynomial over $A$} is a polynomial $p(X)\in R_A[X]$ whose coefficients all lie in $A$.
Let $B/A$ be an extension of sesquiads.
An element $b\in B$ is \e{algebraic} over $A$, if there exists a non-constant polynomial $p$ over $A$, such that   $p(b)=0$ holds in $R_B$.

In that case, division with remainder shows that there exists a unique polynomial $q(X)$ in $R_B[X]$ such that $p(X)=(X-b)q(X)$.
We say that $b$ is \e{separable over $A$}, if we can choose $p(X)$ in a way that $q(b)\ne 0$.
The extension $B/A$ is called \e{separable}, if every algebraic element is separable.

Note that if $A$ is simple, every algebraic element of $B$ is integral over $R_A$.
\end{definition}

\begin{proposition}Let $A$ be a sesquiad.
\begin{enumerate}[\rm (a)]
\item $A$ is simple if and only if the natural map $R_A\to R_{A_\Delta}$ is bijective.
\item If $A$ is simple and $p$ is a polynomial over $A$ of degree $n\ge 1$, then the equation $p(x)=0$ has at most $n$ solutions in $A$.
\end{enumerate}
\end{proposition}

\begin{proof}
(a) Let $A$ be simple and let $a,b\in a$ with $0\ne a\ne b\ne 1$.
We claim that $a-b$ is a unit in $R_A$.
Suppose it's not, then $(a-b)R_A$ is a proper ideal and $A\to R_A/(a-b)R_A$ is not injective, a contradiction. Therefore, $R_A=R_{A_\Delta}$. 
For the converse assume that the natural map $R_A\to\R_{A_\Delta}$ is bijective.
Let $0\ne a\ne b\ne 1$ in $A$ and let $C$ be a congruence on $A$ such that $a\sim_C b$. Then the ideal $I(C)$ contains $a-b$, which is a unit in $R_A$, hence $R_{A/C}=R_A/I(C)=0$, so $C$ is the trivial congruence.

(b) Let $a$ be a solution of $p(x)=0$.
By polynomial division, we find $p(x)=(x-a)q(x)$ for some polynomial $q\in R_A[x]$ of one degree less. If $a'$ is a second solution, then $(a'-a)q(a')=0$ leads to $q(a')=0$ as $(a'-a)$ i invertible in $R_{A_\Delta}$.
Therefore, given $n$ distinct solutions $a_1,\dots a_n$ of $p(x)=0$, we can iterate the process to get $p(x)=\al (x-a_1)\cdots (x-a_n)$.
Let now $a$ be any solution, then $0=\al (a-a_1)\cdots(a-a_n)$.
If $a\ne a_j$ for all $j$, then this product is invertible in $R_{A_\Delta}$, a contradiction. 
\end{proof}

\begin{definition}
We say that a sesquiad $A$ is \e{algebraically closed}, if every non-constant polynomial $p$ over $A$ has a zero in $A$.

Note that every simple sesquiad $A$ can be mapped injectively into an algebraically closed simple sesquiad $B$.
It is, however, not clear whether it can be embedded into such, i.e., if $B$ can be chosen in a way that the ensuing map $R_A\to R_B$ is injective, too.
Also, $B$ is not unique, so there does not exist an "algebraic closure" for sesquiads.
\end{definition}

\begin{lemma}
Let $A$ be a sesquiad and $E\in\spec_cA$. We denote the congruence
$$
\ker\(A_E\to S_E^{-1}R_A/S_E^{-1}I(E)\)
$$
on $A_E$ again by $E$. Then there is a natural isomorphism
$$
A_E/E\cong (A/E)_\Delta.
$$
\end{lemma}

\begin{proof} This is Lemma 3.1.2 of \cite{congruence}.
\end{proof}

\begin{definition}
A morphism $f:X\to Y$ of congruence schemes is called \e{unramified}, if it is locally of finite presentation and for every $x\in X$ the induced map
$$
\CO_{Y,f(x)}/E_{f(x)}\to \CO_{X,x}/E_x
$$
is a finite separable embedding.

The morphism $f$ is called \e{etale}, if it is flat and unramified.
\end{definition}

Note that in particular, a morphism of sesquiads $\ph:A\to B$ is  etale if and only if
\begin{enumerate}[\rm (a)]
\item $R_B$ is a flat, finitely presented $R_A$-algebra, and
\item all localizations $A_{\ph^*E}\to B_E$, $E\in\spec_cB$ are unramified, which means that
$$
A_{\ph^*E}/\ph^*E\to B_E/E
$$
is a finite separable embedding.
\end{enumerate}

\section{Sheaf cohomology}
Recall the notion of a congruence scheme \cite{congruence}.
This is a topological space $X$ together with a sheaf $\CO_X$ of sesquiads which is locally of the form $(\spec_cA,\CO_A)$ for sesquiads $A$.

\begin{definition}
We define a \e{module sheaf} to be a sheaf $\CF$ of pointed sets together with the structure of an $\CO_X(U)$-module on the set $\CF(U)$ for each open set $U\subset X$ such that for any open sets $U\subset V\subset X$ the restriction morphism $\CF(V)\to \CF(U)$ is compatible with the $\CO_X(V)$ and $\CO_X(U)$ module structures.
More precisely, for $m\in \CF(V)$ and $a\in\CO_X(V)$ we insist that $(am)|_U=a|_Um|_U$. 
A \e{morphism} $\ph:\CF\to\CG$ of module sheaves is a morphism of sheaves of pointed sets compatible with the module structures.
It induces module homomorphisms on the stalks $\ph_x:\CF_x\to\CG_x$, $x\in X$.

We say that the morphism $\ph$ is \e{full}, if for every open $U\subset X$ the morphism of modules 
 $\ph_U:\CF(U)\to\CG(U)$ is a full morphism.
\end{definition}

\begin{lemma}
A morphism $\ph:\CF\to\CG$ is full if and only if all stalks $\ph_x:\CF_x\to\CG_x$, $x\in X$ are full.
\end{lemma}

\begin{proof}
Suppose $\ph$ is full and let $x\in X$.
We have to show that $M_{\ph_x(\CF_x)}\cap\CG_x$ is a subset of $\ph_x(\CF_x)$.
So let $s_x$ be an element of $M_{\ph_x(\CF_x)}\cap\CG_x$.
There exists an open neighborhood $U$ of $x$ such that $s_x$ in induced by some $s\in M_{\ph_U(\CF(U))}\cap\CG(U)$.
As $\ph_U$ is full, we have $s\in \ph_U(\CF(U))$, inducing an element of $\ph_x(\CF_x)$.

For the converse direction suppose that all stalks of $\ph$ are full and let $s\in M_{\ph_U(\CF(U))}\cap\CG(U)$ for some open set $U\subset X$.
Since all stalks are full, there is a covering $U=\bigcup_iU_i$ by open sets $U_i$ such that $s|_{U_i}\in \ph_{U_i}(\CF(U_i))$.
On the other hand, as $s\in M_{\ph_U(\CF(U))}$, there are $k_j\in\Z$ and $s_j\in\ph_U(\CF(U))$ such that $s=\sum_{j=1}^nk_js_j$.
Fix $t_j\in \CF(U)$ with $s_j=\ph_U(t_j)$.
It follows that $\sum_{j=1}^nk_jt_j|_{U_i}\in\CF(U_i)+\ker M_{\ph}(U_i)$.
The sum on the right hand side represents a sheaf, therefore $\sum_{j=1}^nk_jt_j\in\CF(U)+\ker M_\ph(U)$ and applying $M_\ph$ we get $s=\sum_{j=1}^nk_js_j\in\ph_U(\CF(U)$.
\end{proof}

Let $\Mod_\full(X)$ denote the category of $\CO_X$-modules and full morphisms.

\begin{proposition}
The category $\Mod_\full(X)$ is belian and contains enough injectives.
\end{proposition}

\begin{proof}
The category is pointed by the zero module.
We show that it is balanced.
Let $\ph:\CF\to\CG$ be epic and monic.
Considering skycraper sheaves we find that for every given $x\in X$ the stalk $\ph_x:\CF_x\to\CG_x$ is epic and monic in the category $\Mod_\full(\CO_{X,x})$.
The latter is balanced by Proposition \ref{prop1.8}, therefore the stalk $\ph_x$ is an isomorphism.
By standard arguments (c.f. \cite{Harts},  Proposition II 1.1), $\ph$ is an isomorphism.

The category $\Mod_\full(X)$ contains finite products, kernels and cokernels.
Let finally $\ph:\CF\to\CG$ have zero cokernel.
Then each stalk $\ph_x$ has zero cokernel, hence is epic.
It follows that $\ph$ is epic and so $\Mod_\full(X)$ is belian.

To see that it has enough injectives, it suffices to consider products of skyscraper sheaves with injective stalks as in \cite{belian}.
\end{proof}

In order to define sheaf cohomology for congruence schemes along the lines of \cite{belian}, we need to show that there exists an ascent datum for the global sections functor
$$
\Ga:\Mod_\full(X)\to\Mod_\full\(\CO_X(X)\).
$$
For this we have to recall some definitions from the paper \cite{belian}.
\begin{itemize}
\item Recall that an \e{ascent functor} $\Asc$ on a belian category $\CB$ is a pointed functor from $\CB$ to some abelian category $\CC$ which is strong-exact and maps epimorphisms to epimorphisms.
\item For a belian category $\CB$, an \e{injective class} is a class $\CI$ of injective objects such that every object in $\CB$ injects into some object of $\CI$ and $\CI$ is closed under finite products.
\item Let $F:\CB\to\CB'$ be a left strong-exact functor of belian categories. An \e{ascent datum} for $F$ is a quadruple $(\CI,\Asc,\Asc',\tilde F)$ consisting of an injective class $\CI$ in $\CB$ and an ascent functor $\Asc:\CB\to \CC$ which maps objects of $\CI$ to injective objects as well an an ascent functor $\Asc'\CB'\to\CC'$ such that on the full subcategory with object class $\CI$ the diagram of functors
$$
\xymatrix{
\CC\ar[r]^{\tilde F}&\CC'\\
\CI\ar[u]^{\Asc}\ar[r]^F&\CB'\ar[u]_{\Asc'}
}
$$
commutes up to isomorphism of functors.
\end{itemize}

Let now $X$ be a congruence scheme and let $\CR$ be the sheaf of rings such that for each open $U\subset X$ the ring $\CR(U)$ is the universal ring of $\CO_X(U)$.
The functor $\Asc:\Mod_\full(X)\to \Mod(\CR)$ mapping a sheaf $\CF$ to the corresponding sheaf $M_\CF$ of $\CR$-modules is clearly an ascent functor.

We define the ascent functor $\Asc':\Mod_\full\(\CO_X(X)\)\to\Mod\(\CR(X)\)$ by mapping $(S,M_S)$ to $M_S$.
Finally we let $\tilde F:\Mod(\CR)\to\Mod\(\CR(X)\)$ be the global sections functor and we let $\CI$ be the class of products of skyscraper sheaves with injective stalks.

\begin{lemma}
The quadruple $(\CI,\Asc,\Asc',\tilde F)$ is an ascent datum for the global sections functor $\Mod_\full(X)\to\Mod_\full\(\CO_X(X)\)$.
\end{lemma}

\begin{proof}
By Lemma 1.9, the ascent functor sends $\CI$ to injectives.
The diagram commutes even without applying an isomorphic of functors.
\end{proof}

Now the results of Section 10 and 11 of \cite{belian} apply to modules of congruence schemes.
So to any $\CO_X$-module $\CF$ we can attach cohomology modules $H^p(X,\CF)$ which are modules under the sesquiad $\CO_X(X)$.

\begin{itemize}
\item Injective sheaves are flabby. Flabby sheaves are acyclic. The cohomology can be computed with arbitrary acyvlic resolutions.
\item If $X$ is noetherian of dimension $n$, then any cohomology in degrees $>n$ vanishes.
\item There is a natural base change of modules $\CF\mapsto\CF_\Z$ which attaches to an $X$-module a $X_\Z$-module and a base change of cohomology such that there is a natural injection
$$
H^p(X,\CF)_\Z\hookrightarrow H^p(X_\Z,\CF_\Z).
$$
\end{itemize}

{\small Mathematisches Institut,
Auf der Morgenstelle 10,
72076 T\"ubingen,
Germany\\
\tt deitmar@uni-tuebingen.de}

\end{document}